\documentclass[12pt]{amsart}
\usepackage{txfonts}
\usepackage{amscd,amssymb}
\usepackage[colorlinks,plainpages,backref,urlcolor=blue]{hyperref}
\usepackage{verbatim}
\usepackage{xcolor}

\newtheorem{theorem}{Theorem}[section]
\newtheorem{corollary}[theorem]{Corollary}
\newtheorem{lemma}[theorem]{Lemma}
\newtheorem{prop}[theorem]{Proposition}

\theoremstyle{definition}
\newtheorem{definition}[theorem]{Definition}
\newtheorem{example}[theorem]{Example}
\newtheorem{remark}[theorem]{Remark}

\newcommand{\N}{\mathbb{N}}
\newcommand{\Z}{\mathbb{Z}}

\newcommand{\C}{\mathbb{C}}

\newcommand{\p}{\mathbb{P}}

\DeclareMathAlphabet{\pazocal}{OMS}{zplm}{m}{n}
\newcommand{\A}{{\pazocal{A}}}

\newcommand{\T}{{\mathcal{T}}}

\newcommand{\cA}{\mathcal{A}}
\newcommand{\cB}{\mathcal{B}}

\newcommand{\cE}{\mathcal{E}}

\newcommand{\cT}{\mathcal{T}}

\newcommand{\cI}{\mathcal{I}}

\newcommand{\cO}{\mathcal{O}}

\newcommand{\TA}{\cT_{\cA}}

\def\dot{\mathchar"013A}
\newcommand{\hdot}{{\raise1pt\hbox to0.35em{\Huge $\dot$}}}

\begin{document}
\date{August, 2023}

\title[A geometric perspective on plus-one generated arrangements of lines]%
{A geometric perspective on plus-one generated arrangements of lines}

\author[A. M\u acinic]{Anca~M\u acinic$^*$}
\address{Simion Stoilow Institute of Mathematics, 
 Bucharest, Romania}
\email{Anca.Macinic@imar.ro}
\thanks{$^*$ Partially supported by 
 a grant of the Romanian Ministry of Education and Research, CNCS - UEFISCDI, project number PN-III-P4-ID-PCE-2020-2798, within PNCDI III}

\author[J. Vall\`es]{Jean Vall\`es$^{**}$}
\thanks{$^{**}$ Partially supported by 
 Bridges ANR-21-CE40-0017}
\address{Université de Pau et des Pays de l'Adour, France}
\email{jean.valles@univ-pau.fr}



\begin{abstract} 
We give a geometric characterisation of plus-one generated projective line arrangements that are next-to-free. We present new succinct proofs, via associated line bundles,  for some properties of  plus-one generated projective line arrangements.
\end{abstract}
 
\maketitle

\section{Introduction} 
\label{sec:introduction}

Let $\cA$ be an arrangement of $n$ lines in $\p^2 = \p^2(\C)$, defined as the zero locus of a homogeneous degree $n$ polynomial $f_{\cA}:= \Pi_{H \in \cA} \alpha_H$, where $\alpha_H \in \C[x,y,z]$ is the linear form that defines the line $H$.

Let $\cI_{\cA}$ be the Jacobian ideal of $f_{\cA}$, i.e. the image of the Jacobian map

\begin{equation}
\label{eq: jacobian}
 \cO_{\p^2}^3 \overset{\nabla f_{\cA} } \longrightarrow   \cO_{\p^2}(n-1),
\end{equation}

 where $\nabla f_{\cA} = [\frac{\partial f_{\cA}}{\partial x} \; \frac{\partial  f_{\cA}}{\partial y} \; \frac{\partial  f_{\cA}}{\partial z}]$,  the matrix with entries the partial derivatives of $ f_{\cA}$ with respect to $x, y, z$.
 
 The kernel $\TA$ of $\nabla f_{\cA}$, defined by the short exact sequence:

 $$\begin{CD}
     0 @>>>  \TA @>>>   \cO_{\p^2}^3  @>>>  \cI_{\cA}  @>>>0
    \end{CD}$$ 
 
  \noindent is a rank $2$ reflexive sheaf, hence a vector bundle over $\p^2$, to which we will refer to as {\it the vector bundle associated to $\cA$}. 
  It is isomorphic to the sheafification of the graded module of Jacobian syzygyes of $f_{\A}$.
  
\begin{definition}
\label{def:free}
The arrangement $\cA$ is called {\it free} with exponents $(a,b), \; a,b \in \N$ if  
$$\TA = \cO_{\p^2}(-a)\oplus \cO_{\p^2}(-b)$$
\end{definition}

 To an arrangement of hyperplanes in general one can naturally associate the poset of intersections of various subsets of its set of hyperplanes, ordered by reversed inclusion, which proves to be in fact a geometric lattice (see \cite{OT}) for details), called {\it the intersection lattice of the arrangement}. Terao conjectured in \cite{OT} that, if an arrangement is free, then all the other arrangements in the realisation space of its intersection lattice are free. 
  
  Motivated by the long standing Terao conjecture,  the study of free arrangements is a very active area of research, and, in connection to that, a series of freeness-like notions emerged in the recent literature. We will refer in this note to Abe's recently introduced notion of plus-one generated arrangements from \cite{A0}.  They appear in subsequent papers on generalized deletion-addition (\cite{A2, ADen}), respectively deletion-restriction problems (\cite{A1}).\\
 
  Since we will only work with arrangements of lines in $\p^2$,  we will call them simply arrangements, considering the context implicit.
  
  \begin{definition}
  \label{def:pog}
The arrangement $\cA$  is called {\it plus-one generated (POG)} of exponents $(a,b)$ and level $d$ if its associated vector bundle $\TA$ has a resolution of type 

\begin{equation}
\label{eq:res}
  \begin{CD}
     0 @>>>  \cO_{\p^2}(-1-d)@>>> \cO_{\p^2}(-d)\oplus \cO_{\p^2}(-b)\oplus \cO_{\p^2}(-a) @>>>  \TA @>>>0
    \end{CD}
 \end{equation}
  \end{definition}
  
  
  We consider here the exponents of a POG arrangement to be ordered, i.e. $a \leq b$.
Notice that, if if $d=b$,  then \ref{def:pog} restricts to the definition of  the nearly free arrangements introduced by Dimca-Sticlaru in \cite{DS}. Also, we have that $c_1(\TA)=1-a-b$, where $c_1(\TA)$ is the first Chern class of $\TA$.\\

Our interest in plus-one generated arrangements is justified by their occurrence in the vicinity of free arrangements, in the following sense.

\begin{theorem} \cite{A0}
\label{thm:Abe_NTF} Let $\A$ be a free arrangement. Then:
\begin{enumerate}
\item For any $H \in \cA, \; \cA \setminus \{ H \}$ is either free or plus-one generated.
\item For any $H \in \p^2, \; \cA \cup \{ H \}$ is either free or plus-one generated.
\end{enumerate}
\end{theorem}

\begin{definition}
\label{def:NTF}
An arrangement is called {\it next to free (NT-free)} if it can be obtained either by deletion  of a line from a free arrangement or by addition of a line to a free arrangement. In the first situation we call the arrangement {\it next to free minus (NT-free minus)}, whereas in the second  situation we call the arrangement {\it next to free plus (NT-free plus)}.
\end{definition}

 Theorem \ref{thm:Abe_NTF} states in particular that NT-free arrangements are either free or plus-one generated. One could naturally ask when is a plus-one generated arrangement also an NT-free one. In \cite[Theorem 1.11]{A0} the author implicitly gives conditions 
for a plus-one generated arrangement to be NT-free, in terms of exponents and combinatorics.
We will give in Theorem \ref{thm:NTF-POG} a geometric characterisation of the situation when a plus-one generated arrangement is NT-free.\\

In section \ref{sec:num_prop} we prove some deletion results for plus-one generated arrangements, and we
revisit and present new simplified geometric proofs, using vector bundles, for a number of  results from \cite{AIM}, see Theorem \ref{thm:AIM4.4} and Proposition \ref{prop:h_val}.

\section{Deletion for plus-one generated line arrangements}
\label{sec:num_prop}

Let $\cA$ be a plus-one generated arrangement of exponents $(a,b)$ and level $d$.
From Definition \ref{def:pog} if follows that $d\ge b$. \\

The resolution \eqref{eq:res} induces a non-zero section:
\begin{equation} \label{min-section}
 \begin{CD}
     0 @>>>  \cO_{\p^2}(-a)@>>> \TA @>>>  I_Z(1-b) @>>>0
    \end{CD}
\end{equation}
where $Z$ is a finite scheme of length $d+1-b$ defined by a complete intersection of a line $l_0 = l_0^{\cA}$ and a degree $d+1-b$ curve:
  $$\begin{CD}
     0 @>>>  \cO_{\p^2}(-1-d)@>>> \cO_{\p^2}(-d)\oplus \cO_{\p^2}(-b) @>>>  I_Z(1-b) @>>>0.
    \end{CD}$$
    \\
    This line $l_0^{\cA}$ will play an important role in formulating necessary and sufficient conditions for a plus-one generated arrangement to be NT-free, see Theorem \ref{thm:NTF-POG}.\\
    
To state the next theorem, we need to recall a classic result on vector bundles.
Given a rank $2$ vector bundle $\cE$ over $\p^2$ and an arbitrary line $l \in \p^2$, the restriction of $\cE$ to $l$ splits as a sum of two line bundles, by Grothendieck's splitting theorem:
$$
\cE|_{l} := \cE \otimes \cO_l = \cO_l(\alpha)\oplus \cO_l(\beta)
$$
where the pair $(\alpha, \beta) \in \Z^2$ is called {\it the splitting type} of $\cE$ on $l$.

\begin{theorem}
\label{thm:AIM4.4}
Let $\cA$ be a plus-one generated arrangement of exponents $(a,b)$ and level $d$ and $l \in \p^2$ arbitrary.
 \begin{enumerate}
   \item If  $l\cap Z=\emptyset$ then $\TA\otimes \cO_l=\cO_l(-a)\oplus \cO_l(1-b)$.
   \item If  $|l\cap Z|=1$ then $\TA\otimes \cO_l=\cO_l(1-a)\oplus \cO_l(-b)$.
   \item If $l=l_0^{\cA}$,  i.e. $|l\cap Z|=d+1-b$, then $\TA\otimes \cO_l=\cO_l(-d)\oplus \cO_l(d+1-a-b)$.
    \end{enumerate}
\end{theorem}

\begin{proof} Tensoring the exact sequence (\ref{min-section}) by $\cO_l$, where $l\subset \p^2$ is a line, we get
\begin{equation} \label{min-section-restricted}
 \begin{CD}
   0 @>>>  \cO_{l}(-a)@>>> \TA\otimes \cO_l @>>>  \cO_l(1-b-|l\cap Z|)\oplus \cO_{l\cap Z} @>>>0.
    \end{CD}
\end{equation}

  There are three different cases for a line $l$ meeting $Z$: $l$ does not meet $Z$, then $|l\cap Z|=0$; $l$ cuts transversally $l$ and meets $Z$, then $|l\cap Z|=1$; or $l=l_0^{\cA}$, then $l\cap Z=Z$ and $|l\cap Z|=d+1-b$. These three cases
  give the three different splitting types of $\TA$ along  $l$.
\end{proof}

\begin{corollary}
\label{cor:l_0}
Let $\cA$ be a plus-one generated arrangement of exponents $(a,b)$ and level $d>b$. Then there exists a unique line $l^{\cA}_0 \subset \p^2$ such that the splitting type of $\TA$ on  $l^{\cA}_0$ is $(a+b-d-1,d)$.
\end{corollary}

\begin{remark}
\label{rem:ineq_d}
\begin{enumerate}
\item Theorem \ref{thm:AIM4.4} retrieves \cite[Theorem 4.4]{AIM} and extends similar results for nearly free arrangements from \cite{AD,MV}.
\item Considering the last splitting type one can also deduce that $d+1\le a+b$. Indeed since $\TA$ is the kernel of the Jacobian map \eqref{eq: jacobian},  which restricted to $l$ remains exact, then $\TA\otimes \cO_l$ cannot have a strictly positive component.
\end{enumerate}
\end{remark}

\medskip

Let $l$ be a line in $\cA$ and denote by  $\cA \setminus l$ the arrangement obtained from $\cA$ by removing $l$. We first recall the well known relation:
\begin{lemma}
Let 
$h:=|l\cap \cA|$ be the number of distinct intersection points  and $t$ be the number of triple points of $\cA$ on $l$ counted with multiplicity. Then $t=|\cA|-h-1$.
\end{lemma}

We have also two canonical exact sequences according to the data $l,\cA,\cA\setminus l$ and $t$ the number of triple points of $\cA$ on $l$:
\begin{equation}\label{AtoA-l}
\begin{CD}
     0 @>>>  \TA@>>> \cT_{\cA \setminus l} @>>>  \cO_l(-t)=\cO_l(h+1-a-b) @>>>0
    \end{CD}
\end{equation}

and after dualizing this exact sequence we get 
$$ \begin{CD}
     0 @>>> \cT_{\cA \setminus l}^{\vee}  @>>> \TA^{\vee} @>>>  \cO_l(a+b-h) @>>>0.
    \end{CD}$$
Since $ \cT_{\cA \setminus l}^{\vee}= \cT_{\cA \setminus l}(a+b-2)$ and $\TA^{\vee}=\TA(a+b-1)$ we obtain after shifting by $1-a-b$:
\begin{equation}\label{A-ltoA}
\begin{CD}
     0 @>>> \cT_{\cA \setminus l}(-1)  @>>> \TA @>>>  \cO_l(1-h) @>>>0.
    \end{CD}
\end{equation}
Now these exact sequences force $h$ to be one of the following numbers, giving a short geometric argument for \cite[Proposition 4.7]{AIM}:

\begin{prop}
\label{prop:h_val}
The allowed values for $h$ are: \begin{enumerate}
\item $h<a$;
\item $h=a$;
\item $h=a+1$;
\item $h=b$;
\item $h=b+1$;
\item $h=d+1$.
\end{enumerate}
\end{prop}
 \begin{proof}
The surjective map $$\begin{CD}
      \TA @>>>  \cO_l(1-h) 
    \end{CD}$$ induces a surjective map: 
    $$\begin{CD}
      \TA\otimes \cO_l @>>>  \cO_l(1-h).
    \end{CD}$$
    When $\TA\otimes \cO_l=\cO_l(-a)\oplus \cO_l(1-b)$ the allowed values for $h$ are $h\le a$, 
    $h=a+1$ or $h=b$. Other values would not give a surjection. 
    
\smallskip

When $\TA\otimes \cO_l=\cO_l(1-a)\oplus \cO_l(-b)$ the allowed values for $h$ are $h< a$, 
    $h=a$ or $h=b+1$. Other values would not give a surjection. 
    
\smallskip

When $\TA\otimes \cO_l=\cO_l(-d)\oplus \cO_l(d+1-a-b)$ the allowed values for $h$ are $h=d+1$, 
     or $h=a+b-d$ or $h\le min(d,a+b-d-1)$. Other values would not give a surjection. Since $d \geq b \geq a$, we get in the last two cases $h \leq a$, respectively $h<a$. 
 \end{proof}
 
 For lines $l \in \cA$ exhibiting some of the values of $h$ from Proposition \ref{prop:h_val}, one can precisely describe the arrangement obtained by deletion of the line $l$ from $\cA$. 
 
\begin{prop} 
\label{prop:a+1}
Let $\cA$ be a plus-one generated of type $(a,b)$ and level $d, \; a < d$.
Let $l\in \cA$ and $h=|l\cap \cA|=a+1$. Then 
$\cA\setminus l$ is plus-one generated of type $(a,b-1)$ and level $(d-1)$.
\end{prop}
 
\begin{proof}
 We have according to (\ref{AtoA-l}):
 $$\begin{CD}
     0 @>>> \cT_{\cA \setminus l}(-1)  @>>> \TA @>>>  \cO_l(-a) @>>>0
    \end{CD}$$
 and a commutative diagram:
 
   $$\begin{CD}
    @. 0@>>> \cO_{\p^2}(-b)\oplus \cO_{\p^2}(-d) @>>>\cT_{\cA \setminus l}(-1)   \\
   @.@VVV @VVV @VVV \\
     0@>>> \cO_{\p^2}(-d-1)@>>>  \cO_{\p^2}(-b)\oplus \cO_{\p^2}(-d)\oplus \cO_{\p^2}(-a)  @>>> \TA @>>>0\\
      @. @VVV @VVV @VVV \\
    0 @>>> \cO_{\p^2}(-a-1)  @>>> \cO_{\p^2}(-a) @>>>  \cO_l(-a) @>>>0
    \end{CD}$$ 
which implies using the snake lemma:
$$\begin{CD}
     0 @>>> \cO_{\p^2}(-b)\oplus \cO_{\p^2}(-d)   @>>> \cT_{\cA \setminus l}(-1) @>>>  \cO_{\Gamma}(-a-1) @>>>0
    \end{CD}$$
where $\mathrm{deg}(\Gamma)=d-a$. We then deduce

 $$\begin{CD}
    @. 0@>>> \cO_{\p^2}(-d-1) @=\cO_{\p^2}(-d-1)   \\
   @.@VVV @VVV @VVV \\
     0@>>> \cO_{\p^2}(-b)\oplus \cO_{\p^2}(-d)@>>>  \cO_{\p^2}(-b)\oplus \cO_{\p^2}(-d)\oplus \cO_{\p^2}(-a-1)  @>>> \cO_{\p^2}(-a-1) @>>>0\\
      @. @| @VVV @VVV \\
    0 @>>> \cO_{\p^2}(-b)\oplus \cO_{\p^2}(-d)  @>>> \cT_{\cA \setminus l}(-1)  @>>>  \cO_{\Gamma}(-a-1) @>>>0.
    \end{CD}$$ 
 \end{proof}
    
\begin{prop} 
\label{prop:b+1}
Let $\cA$ be a plus-one generated of type $(a,b)$ and level $d>b$.
Let $l\in \cA$ and $h=|l\cap \cA|=b+1$. Then 
$\cA\setminus l$ is plus-one generated of type $(a-1,b)$ and level $(d-1)$.
\end{prop}
 
 \begin{proof}
 We have according to (\ref{AtoA-l}):
 $$\begin{CD}
     0 @>>> \cT_{\cA \setminus l}(-1)  @>>> \TA @>>>  \cO_l(-b) @>>>0
    \end{CD}$$
 and a commutative diagram:
 
   $$\begin{CD}
    @. 0@>>> \cO_{\p^2}(-a)\oplus \cO_{\p^2}(-d) @>>>\cT_{\cA \setminus l}(-1)   \\
   @.@VVV @VVV @VVV \\
     0@>>> \cO_{\p^2}(-d-1)@>>>  \cO_{\p^2}(-b)\oplus \cO_{\p^2}(-d)\oplus \cO_{\p^2}(-a)  @>>> \TA @>>>0\\
      @. @VVV @VVV @VVV \\
    0 @>>> \cO_{\p^2}(-b-1)  @>>> \cO_{\p^2}(-b) @>>>  \cO_l(-b) @>>>0
    \end{CD}$$ 
which implies using the snake lemma:
$$\begin{CD}
     0 @>>> \cO_{\p^2}(-a)\oplus \cO_{\p^2}(-d)   @>>> \cT_{\cA \setminus l}(-1) @>>>  \cO_{\Delta}(-b-1) @>>>0
    \end{CD}$$
where $\mathrm{deg}(\Delta)=d-b$. We then deduce

 $$\begin{CD}
    @. 0@>>> \cO_{\p^2}(-d-1) @=\cO_{\p^2}(-d-1)   \\
   @.@VVV @VVV @VVV \\
     0@>>> \cO_{\p^2}(-a)\oplus \cO_{\p^2}(-d)@>>>  \cO_{\p^2}(-1-b)\oplus \cO_{\p^2}(-d)\oplus \cO_{\p^2}(-a)  @>>> \cO_{\p^2}(-b-1) @>>>0\\
      @. @| @VVV @VVV \\
    0 @>>> \cO_{\p^2}(-a)\oplus \cO_{\p^2}(-d)  @>>> \cT_{\cA \setminus l}(-1)  @>>>  \cO_{\Gamma}(-b-1) @>>>0.
    \end{CD}$$ 
\end{proof}

\begin{prop}
\label{prop:d+1}
 Let $\cA$ be a plus-one generated of type $(a,b)$ and level $d$.
Let $l\in \cA$ and $h=|l\cap \cA|=d+1$. Then 
$\cA\setminus l$ is free with exponents $(a-1,b-1)$.
\end{prop}

\begin{proof}
  We have according to (\ref{AtoA-l}):
 $$\begin{CD}
     0 @>>> \cT_{\cA \setminus l}(-1)  @>>> \TA @>>>  \cO_l(-d) @>>>0
    \end{CD}$$
 and a commutative diagram:
 
   $$\begin{CD}
    @. 0@>>> \cO_{\p^2}(-a)\oplus \cO_{\p^2}(-b) @>>>\cT_{\cA \setminus l}(-1)   \\
   @.@VVV @VVV @VVV \\
     0@>>> \cO_{\p^2}(-d-1)@>>>  \cO_{\p^2}(-b)\oplus \cO_{\p^2}(-d)\oplus \cO_{\p^2}(-a)  @>>> \TA @>>>0\\
      @. @| @VVV @VVV \\
    0 @>>> \cO_{\p^2}(-d-1)  @>>> \cO_{\p^2}(-d) @>>>  \cO_l(-d) @>>>0
    \end{CD}$$ 
which implies using the snake lemma:
$$\begin{CD}
     0 @>>> \cO_{\p^2}(-a)\oplus \cO_{\p^2}(-b)   @= \cT_{\cA \setminus l}(-1).
    \end{CD}$$
\end{proof}

\begin{remark}
\label{rem:consequence}
Proposition \ref{prop:d+1} is also a consequence of \cite[Theorem 1.11]{A0}. 
\end{remark}

If $h  \leq a$ then $\cA \setminus l$ is not necessarily free or plus-one generated.

\begin{example}
\label{ex:cntr}
Let $\cA$ be the arrangement of equation $xyz(x+y)(x-y)(x+4y+z)(y+z)=0$. Then $\cA$ is a plus-one generated arrangement of exponents $(3,4)$ and level $5$. 

\noindent The line $l$ of equation $x+4y+z=0$ intersects generically the rest of the lines in the arrangement, so $h=6=d+1$ in this case. Then $\cA$ with this line deleted gives a free arrangement of exponents $(2,3)$, see for instance Proposition \ref{prop:d+1}.
 
\noindent If $l$ is one of the lines of equations $x-y=0, x+y=0, x=0$, we have $h=4=a+1$. If we delete from $\cA$ any  one of these lines we get, by Proposition \ref{prop:a+1}, a plus-one generated arrangement of exponents $(3,3)$ and level $4$.

\noindent For any of the lines of equations $z=0, y+z=0$ we have  $h=5=b+1$. If we delete from $\cA$ any of these lines we get, by Proposition \ref{prop:b+1}, a plus-one generated arrangement of exponents $(2,4)$ and level $4$.

\noindent For the line $l:y=0$ we have $h=3=a$. Then one easily checks (using for instance Macaulay2) that  the arrangement obtained by deleting the line $l$ from $\cA, \; \cA \setminus l$, is neither free nor plus-one generated, since its associated derivation module has $5$ generators.
\end{example}

\section{A geometric characterisation of NT-free arrangements of lines}

 Given an arrangement $\cA$, for each $H \in \cA$ one has an associated multiarrangement, the Ziegler restriction of $\cA$ onto $H$, as introduced  in \cite{Z}. In our context, where $\cA$ is a complex projective line arrangement, the Ziegler restrictions are multiarrangements in $\C^2$, and their associated graded module of derivations is always free, of rank $2$ (see for instance \cite{AD, Y} for details). The exponents of a Ziegler restriction are by definition the  pair of degrees of the generating set of derivations for this graded module. 
 
 Understanding exponents of Ziegler restrictions turned out to be an essential tool in the study of freeness of arrangements, as proved by Yoshinaga in \cite{Y}. Notably, the splitting type of the vector bundle associated to an arrangement $\cA$  onto a line $l \in \cA$ coincides to  the exponents of the Ziegler restriction of  $\cA$  onto $l$, again by \cite{Y}.\\

We have the following generalization of an addition-type formula from exponents of  Ziegler restrictions to splitting types.

\begin{prop}
\label{prop:add_split}
Let  $\cA, \;  \cB$ be two line arrangements such that $\cB = \cA \cup \{H\}$, for some line $H\subset \p^2 $.  Take $l \subset \p^2$ a line such that $l \notin \cB$ and 
denote by $(a^{\cA}, b^{\cA})$ the splitting type along the line $l$ for $\TA$ and by 
$(a^{\cB}, b^{\cB})$ the splitting type along the line $l$ for $\T_{\cB}$.
Then 
\begin{equation}
\label{eq:add}
 (a^{\cB}, b^{\cB}) \in \{(a^{\cA}+1, b^{\cA}), \; (a^{\cA}, b^{\cA}+1)\}
\end{equation} 
\end{prop}

\begin{proof}
We always have the exact sequence:
$$\begin{CD}
0 @>>> \T_{\cA \cup \{H\}}  @>>> \TA  @>>> \cO_H(-t) @>>> 0 
\end{CD}$$
where $t$ is the number of triple points on $H$ of $\cA \cup \{H\}$.

Restricting to $\cO_l$ we obtain :
$$\begin{CD}
0 @>>> \T_{\cA \cup \{H\}} \otimes \cO_l  @>>> \TA  \otimes \cO_l  @>>> \cO_{p} @>>> 0 
\end{CD}$$

\noindent where  $p=l \cap H$. So, if 
$$ \TA  \otimes \cO_l = \cO_l(-a^{\cA}) \oplus \cO_l(-b^{\cA}),$$

\noindent then we necessarily have  
 $$ \cT_{\cB}  \otimes \cO_l = \cO_l(-a^{\cA}-1) \oplus \cO_l(-b^{\cA})$$ or  
 $$ \cT_{\cB}  \otimes \cO_l = \cO_l(-a^{\cA}) \oplus \cO_l(-b^{\cA}-1).$$
\end{proof}

\begin{remark}
\label{rem:add}
\begin{enumerate}
\item If we change the hypothesis of  Proposition \ref{prop:add_split} by assuming $l$ to be a line in $\cA$, then the splitting types are exponents of Ziegler restrictions, and the conclusion of the proposition still holds, see \cite{WY},  \cite{AN}.
\item For $l=H$, the claim of Proposition \ref{prop:add_split} no longer holds, as we can see in the following counterexample.
\end{enumerate}
\end{remark}

\begin{example}
\label{ex:counter}
Let $\cB$ be an arrangement with precisely two multiple points $P, Q$ of multiplicities $p+1$, respectively  $q+1$, with $p,q>2$, and only multiple points of multiplicity $2$ in rest, such that the line $l := PQ \in \cB$. Then $\cB$ is free of exponents $(p, q)$ and  $|\cB| = p+q+1$. 
 $\cA := \cB \setminus \{ l \}$ is plus-one generated with exponents $(p, q)$ and level $p+q-2$.  

The splitting type along the line $l$ for the vector bundle associated to the arrangement $\cB$ is  $(p, q)$ and the splitting type along the line $l$ for the vector bundle associated to the arrangement $\cA$ is $(1,p+q-2)$, hence an equality of type \eqref{eq:add} does not take place.
\end{example}

In particular, the previous example shows that any pair of positive integers can be realized as exponents of a plus-one generated arrangement of lines, compare to \cite{DS1} for similar results on nearly free arrangements.

Recall that, for $\cA$ plus-one generated of exponents  $(a,b)$ and level  $d>b$, there exists a unique line $l^{\cA}_0 \subset \p^2$ as in Corollary \ref{cor:l_0}.
 
 \begin{lemma}
 \label{lemma:+-}
 \begin{enumerate}
\item  Let $\cA$ be plus-one generated of exponents  $(a,b)$ and level  $d>b$, such that $l^{\cA}_0 \notin \cA$. Then $\cA$ cannot be  NT-free plus.
\item  Let $\cA$ be   plus-one generated of exponents  $(a,b)$ and level  $d>b$, such that $l^{\cA}_0 \in \cA$. Then $\cA$ cannot be  NT-free minus.
  \end{enumerate}
  \end{lemma}
 
 \begin{proof}
 {\it Part (1)}
Assume the contrary, that $\cA$ is  NT-free plus.  That is, there is a line $H \in \cA$ such that $\cA \setminus \{H\}$ is free. Then, by  \cite[Thm. 1.11]{A0}, $|\cA \cap H| = d+1$. By \cite{WY} the exponents of the Ziegler restriction of $\cA$ onto $H$ are $(a+b-d-1, d)$, which coincides with the splitting type onto $H$  for the bundle of logarithmic vector fields associated to the arrangement $\cA$. But, since $l^{\cA}_0$ is unique with the property that the splitting type onto $l^{\cA}_0$  for the bundle of logarithmic vector fields associated to the arrangement equals  $(a+b-d-1, d)$,  from Corollary \ref{cor:l_0}, it follows that $l^{\cA}_0 = H$, so $l^{\cA}_0 \in \cA$, contradiction. \\

 {\it Part (2)}
Just as before, assume the contrary, that $\cA$ is  NT-free minus.  That is, assume there exists  a line $H \subset \p^2$ such that $\cB := \cA \cup \{H\}$ is free.  Then $exp(\cB) = (a,b)$. 
 Since the exponents of the Ziegler restriction of $\cA$ onto $l^{\cA}_0$ are $(a+b-d-1, d)$, it follows that the exponents of the Ziegler restriction of $\cB$ onto $l^{\cA}_0$  should be one of the two pairs $\{(a+b-d, d), (a+b-d-1, d+1)\}$ (see Remark  \ref{rem:add}(1)).  At the same time, since $\cB$ is free of exponents $(a,b)$, the exponents of the Ziegler restriction of $\cB$ onto $l^{\cA}_0$  should be equal to  $(a,b)$, but this implies $b \in \{d,d+1\}$, contradiction.
\end{proof}
  
\begin{theorem}
\label{thm:NTF-POG}
Let $\cA$ be plus-one generated of exponents $(a,b)$ and level $d>b$.
\begin{enumerate}
\item Assume  $l^{\cA}_0 \notin \cA$. Then the following are equivalent:
	\begin{enumerate}
	\item $\cA$ is NT-free.
	\item $\cA$ is NT-free minus.
	\item $d = |\cA| - |\cA \cap l^{\cA}_0|$. 
	\item $\cA  \cup \{l^{\cA}_0\}$ is free.
	\end{enumerate}  
\item 	Assume $l^{\cA}_0 \in \cA$. Then the following are equivalent:
	\begin{enumerate}
	\item $\cA$ is NT-free.
	\item $\cA$ is NT-free plus.
	\item $d +1  = |\cA \cap  l^{\cA}_0|$. 
	\item $\cA  \setminus \{l^{\cA}_0\}$ is free.
	\end{enumerate}  
\end{enumerate}	
\end{theorem}

\begin{proof}
{(1) \em Case $l^{\cA}_0 \notin \cA$}

Assume $\cA$ is NT-free. 
Since, by Lemma \ref{lemma:+-}(1), $\cA$ cannot be NT-free plus, it follows that $\cA$ is NT-free if and only if $\cA$ is NT-free minus (i.e. there exists a line $H \subset \p^2$ such that $\cB := \cA \cup \{H\}$ is free). 
From \cite[Thm. 1.11]{A0}, this holds if and only if  $d = |\cA| - |\cA \cap H|$. Moreover, $exp(\cB) = (a,b)$.

To conclude the proof, we only need to show that $H = l^{\cA}_0$.

Assume the contrary,  $H \neq l^{\cA}_0$.  Then $l^{\cA}_0 \notin \cB$. Denote by $(a ^{\cA}, b ^{\cA})$ the splitting type onto $ l^{\cA}_0$ for the vector bundle associated to $\cA$ and by $(a ^{\cB}, b ^{\cB})$ the splitting type onto $ l^{\cA}_0$ for the vector bundle associated to $\cB$.  By Proposition \ref{prop:add_split}, we have:
$$
(a ^{\cB}, b ^{\cB}) \in \{(a ^{\cA}+1, b ^{\cA}), \; (a^{\cA}, b ^{\cA}+1)\}.
$$
 But $(a^{\cB}, b^{\cB}) = (a,b)$ and $(a^{\cA}, b^{\cA}) = (a+b-d-1, d)$, so $b \in \{d,d+1\}$, contradiction. Then necessarily $H=l^{\cA}_0$.  

{(2) \em Case $l^{\cA}_0 \in \cA$}

Assume $\cA$ is NT-free. 
Since, by Lemma \ref{lemma:+-}(2), $\cA$ cannot be NT-free minus, it follows that $\cA$ is NT-free if and only if $\cA$ is  NT-free plus,  i.e. there exists a line $H$ such that $\cB := \cA \setminus \{H\}$ is free.  
According to  \cite[Thm. 1.11]{A0}, this latter claim holds if and only if $ |\cA^H| = d+1$, where $\cA^H$ is the restriction of $\cA$ to $H$. In this case, $\cB$ is free of exponents $(a-1, b-1)$. 

Assume   $H \neq l^{\cA}_0$. To conclude the proof, it is enough to show that this assumption leads to a contradiction. Notice that in this assumption $l^{\cA}_0 \in \cB$.
Since the exponents of the Ziegler restriction of $\cA$ onto $l^{\cA}_0$ are $(a+b-d-1, d)$, it follows that the exponents of the Ziegler restriction of $\cB$ onto $l^{\cA}_0$  should be either $(a+b-d-2, d)$ or $(a+b-d-1, d-1)$. But, since $\cB$ is free of exponents $(a-1, b-1)$,
we get $b \in \{d,d+1\}$, contradiction.
 \end{proof}
 
 \begin{remark}
 \label{rem:NNTF+-}
 If $\cA$ is a plus-one generated arrangement of exponents $(a,b)$ and level $d$, the condition $d>b$ ensures that $\cA$ cannot be simultaneously NT-free minus and NT-free plus (see Lemma \ref{lemma:+-}).
 But there exist plus-one generated arrangements with $b=d$ (i.e. nearly free) that are at the same time NT-free minus and NT-free plus. Take for instance the arrangement $\cA$  of equation $xyz(x+y)(y+z)(x+2y+z)=0$.  $\cA$ is plus-one generated of exponents $(3,3)$ and level $3$. $\cA \setminus \{x=0\}$ is free of exponents $(2,2)$, so $\cA$ is NT-free plus, and $\cA \cup \{y+\frac{1}{2}z = 0 \}$ is free of exponents $(3,3)$, so $\cA$ is also NT-free minus.
 \end{remark}
 
 {\bf Question}: Are there any examples of plus-one generated arrangements that satisfy one of the two conditions below? 
 
 1.  A plus-one generated arrangement $\cA$ of exponents $(a,b)$ and level $d>b$ with $l^{\cA}_0 \in \cA$ such that $|\cA^{ l^{\cA}_0}| \neq d+1$. Notice that this latter condition is equivalent to $|\cA^{ l^{\cA}_0}| < d+1$, by Proposition \ref{prop:h_val}.
 
 2.  A plus-one generated arrangement $\cA$ of exponents $(a,b)$ and level $d>b$ with $l^{\cA}_0 \notin \cA$ such that $ |\cA| - |\cA \cap l^{\cA}_0| \neq d$.

\end{document}